\documentclass[final,leqno]{siamltex}
\usepackage{epsfig}
\usepackage{amsmath}
\usepackage{amssymb}
\usepackage{tikz}
\usepackage{graphicx}
\usepackage[notcite,notref]{showkeys}
\newtheorem{algorithm}{Weak Galerkin Algorithm}
\newcommand{\bq}{{\bf q}}
\newcommand{\bn}{{\bf n}}
\newcommand{\bx}{{\bf x}}
\newcommand{\bv}{{\bf v}}
\newcommand{\bw}{{\bf w}}

\def\T{{\mathcal T}}
\def\E{{\mathcal E}}
\def\Q{{\mathbb Q}}

\def\l{{\langle}}
\def\r{{\rangle}}

\def\bn{{\bf n}}
\def\bq{{\bf q}}

\def\bbQ{\mathbb{Q}}
\newcommand{\pT}{{\partial T}}

\def\3bar{{|\hspace{-.02in}|\hspace{-.02in}|}}

  \def\b#1{\mathbf{#1}} 
\def\a#1{\begin{align*}#1\end{align*}} \def\an#1{\begin{align}#1\end{align}} 
 \def\div{\operatorname{div}}
\def\p#1{\begin{pmatrix}#1\end{pmatrix}}

\title{A stabilizer free weak Galerkin finite element
   method on polytopal mesh: Part II}

\author{Xiu Ye\thanks{Department of
Mathematics, University of Arkansas at Little Rock, Little Rock, AR
72204 (xxye@ualr.edu). This research was supported in part by
National Science Foundation Grant DMS-1620016.}
\and
Shangyou Zhang\thanks{Department of
Mathematical Sciences, University of Delaware, Newark, DE 19716 (szhang@udel.edu).}
}

\begin{document}

\maketitle

\begin{abstract}
A stabilizer free weak Galerkin (WG) finite element method on polytopal mesh has been introduced in Part I of this paper (J. Comput. Appl. Math, 371 (2020) 112699. arXiv:1906.06634.)
Removing  stabilizers from discontinuous  finite element methods  simplifies
formulations and reduces programming complexity. The purpose of this paper is to introduce a new WG method without stabilizers on polytopal mesh that has convergence rates one order higher than optimal convergence rates.
This method is the first WG method that achieves superconvergence on polytopal mesh. Numerical examples in 2D and 3D are presented verifying the theorem.
\end{abstract}


\pagestyle{myheadings}

\section{Introduction}\label{Section:Introduction}

A stabilizing/penalty term  is often used in finite element methods
with discontinuous approximations  to enforce  connection of discontinuous functions
across element boundaries. Removing  stabilizers from discontinuous  finite element method is desirable since it  simplifies
formulation and reduces programming complexity. A stabilizer free weak Galerkin finite element has been developed in \cite{yz-sf-wg} for the following model problem: seek an unknown function $u$ satisfying
\begin{eqnarray}
-\Delta u&=&f\quad \mbox{in}\;\Omega,\label{pde}\\
u&=&0\quad\mbox{on}\;\partial\Omega,\label{bc}
\end{eqnarray}
where $\Omega$ is a polytopal domain in $\mathbb{R}^d$.
The WG method developed in \cite{yz-sf-wg} has the following simple formulation without any stabilizers:
\begin{equation}\label{dfe}
(\nabla_w u_h,\nabla_w v)=(f,v),
\end{equation}
where  $\nabla_w$ is weak gradient. Remove of stabilizing terms from the WG finite element methods is challenging, specially on polytopal mesh.
Construction of  spaces to approximate $\nabla_w$  is the key of maintaining ultra simple formulation (\ref{dfe}). The main idea in \cite{yz-sf-wg} is to raise the degree of polynomials used to compute weak gradient $\nabla_w$. In \cite{yz-sf-wg,cdg2}, gradient is approximated by a polynomial of order $j=k+n-1$ with $n$ the number of sides of  polygonal element.
This result has been improved in \cite{aw,aw1} by reducing the degree of polynomial $j$. In \cite{liu, mu}, Wachspress coordinates \cite{cw} are used to approximate $\nabla_w$, which are usually rational functions, instead of
   polynomials.

In this paper, we introduce a new stabilizer free WG finite element method on polytopal mesh. This method is the first WG method that achieves superconvergence on polytopal mesh.
In this method we use piecewise low order polynomials on a polygonal element to approximate $\nabla_w$ instead of using one piece high order polynomial in \cite{yz-sf-wg}. While the stabilizer free WG method in \cite{yz-sf-wg} has optimal convergence rates, our new WG method improves the convergence rate from optimality by order one in both an energy norm and the $L^2$ norm. Superconvergence results of the WG methods have been investigated in \cite{wy} on simplicial mesh. This method is the first WG method that achieves superconvergence on polytopal mesh, which has been verified theoretically and computationally. Extensive numerical examples are tested for the new WG elements of different degrees in two and three dimensional spaces.

\section{Weak Galerkin Finite Element Schemes}\label{Section:wg-fem}

Let ${\cal T}_h$ be a partition of the domain $\Omega$ consisting of
polygons in two dimension or polyhedra in three dimension satisfying
a set of conditions specified in \cite{wymix}. Denote by ${\cal E}_h$
the set of all edges or flat faces in ${\cal T}_h$, and let ${\cal
E}_h^0={\cal E}_h\backslash\partial\Omega$ be the set of all
interior edges or flat faces. For every element $T\in \T_h$, we
denote by $h_T$ its diameter and mesh size $h=\max_{T\in\T_h} h_T$
for ${\cal T}_h$.

For a given integer $k \ge 0$, let $V_h$ be the weak Galerkin finite
element space associated with $\T_h$ defined as follows
\begin{equation}\label{vhspace}
V_h=\{v=\{v_0,v_b\}:\; v_0|_T\in P_k(T),\ v_b|_e\in P_{k}(e),\ e\subset\pT,  T\in \T_h\}
\end{equation}
and its subspace $V_h^0$ is defined as
\begin{equation}\label{vh0space}
V^0_h=\{v: \ v\in V_h,\  v_b=0 \mbox{ on } \partial\Omega\}.
\end{equation}
We would like to emphasize that any function $v\in V_h$ has a single
value $v_b$ on each edge $e\in\E_h$.

For given $T\in\T_h$ and $v=\{v_0,v_b\}\in V_h$, a weak gradient $\nabla_wv$ is a piecewise polynomial satisfying  $\nabla_w v|_T \in \Lambda_k(T)$ and
\begin{equation}\label{d-d}
  (\nabla_w v, \bq)_T = -(v_0, \nabla\cdot \bq)_T+ \langle v_b, \bq\cdot\bn\rangle_{\partial T}\qquad
   \forall \bq\in \Lambda_k(T),
\end{equation}
where $\Lambda_k(T)$ will be defined in the next section.

For simplicity, we adopt the following notations,
\begin{eqnarray*}
(v,w)_{\T_h} &=&\sum_{T\in\T_h}(v,w)_T=\sum_{T\in\T_h}\int_T vw d\bx,\\
 \l v,w\r_{\partial\T_h}&=&\sum_{T\in\T_h} \l v,w\r_\pT=\sum_{T\in\T_h} \int_\pT vw ds.
\end{eqnarray*}

\begin{algorithm}
A numerical approximation for (\ref{pde})-(\ref{bc}) can be
obtained by seeking $u_h=\{u_0,u_b\}\in V_h^0$
satisfying  the following equation:
\begin{equation}\label{wg}
(\nabla_wu_h,\nabla_wv)_{\T_h}=(f,\; v_0) \quad\forall v=\{v_0,v_b\}\in V_h^0.
\end{equation}
\end{algorithm}

\section{Construction of $\Lambda_k(T)$ and Its Properties}
The space $H(\div;\Omega)$ is defined as the set of vector-valued functions on $\Omega$ which,
together with their divergence, are square integrable; i.e.,
\[
H(\div; \Omega)=\left\{ \bv\in [L^2(\Omega)]^d:\; \nabla\cdot\bv \in L^2(\Omega)\right\}.
\]
We start this section by defining $\Lambda_h(T)$. For any $T\in\T_h$, it can be divided in to a set of disjoint triangles $T_i$ with $T=\cup T_i$.  Then $\Lambda_h(T)$ can be defined as
\begin{eqnarray}
\Lambda_k(T)=\{\bv\in H(\div,T):&&\ \bv|_{T_i}\in RT_k(T),\;\;\nabla\cdot\bv\in P_k(T),\label{lambda}\\
&&\bv\cdot\bn|_e\in P_k(e),\;e\subset\pT\},\nonumber
\end{eqnarray}
where $RT_k(T)$ is the usual Raviart-Thomas element \cite{bf} of order $k$.

\begin{lemma}
Let $\Pi_h : H(\div,\Omega) \to H(\div,\Omega) \cap \otimes \Lambda_k(T)$ be defined in
 \eqref{p-j} below.  For $\bv \in H(\div,\Omega)$ and for all $T\in\mathcal T_h$,
   we have,
\begin{eqnarray}
(\Pi_h\bv,\;\bw)_T&=&(\bv,\;\bw)_T  \quad \forall \bw\in [P_{k-1}(T)]^d,\label{P1}\\
\l \Pi_h\bv\cdot\bn,\; q\r_e &=&\l \bv\cdot\bn,\; q\r_e\quad\forall q\in P_{k}(e),
     e\subset\pT,\label{P2} \\
(\nabla\cdot\bv,\;q)_T&=&(\nabla\cdot\Pi_h\bv,\;q)_T \quad\forall q\in P_k(T), \label{key2}\\
-(\nabla\cdot\bv, \;v_0)_{\T_h}&=&(\Pi_h\bv, \;\nabla_w v)_{\T_h}
    \quad \forall v=\{v_0,v_b\}\in V_h^0, \label{key1}\\
\|\Pi_h\bv-\bv\|&\le& Ch^{k+1}|\bv|_{k+1}.\label{key3}
\end{eqnarray}
\end{lemma}

\begin{proof}  A similar interpolation operator $\Pi_h$ is studied in \cite{Lin}
   which does not require one piece polynomial $\Pi_h \b v$ on one face of $T$, i.e.,
     \eqref{p-j7} is omitted below in the definition of $\Pi_h$.
   The proof here is very similar to the one in
    \cite{Lin}.

  We assume no additional inner vertex/edges is introduced in subdividing
   a polygon/polyhedron $T$ in to $n$ triangles/tetrahedrons $\{T_i\}$.
   That is,  we have precisely $n-1$ internal edges/triangles which separate $T$
    into $n$ parts. For simple notation,  only one face $e_1$ of $T$ is subdivided in to
   $m$ edges/triangles, $e_{1,1}, \dots, e_{1,m}$.  Note that in 3D a non-triangular
     polygonal face has to be subdivided in to
   several triangles.   We limit the proof to 3D.
    We need only omit the fourth equation \eqref{p-j4} in \eqref{p-j} to get
  a 2D proof.

  On $n$ tetrahedrons,  a function of $\Lambda_k$ can be expressed as
 \an{\label{v-e}  \bv_h|_{T_{i_0}} = \sum_{i+j+l\le k}  \p{a_{1,ijl}\\a_{2,ijl}\\a_{3,ijl}} x^i y^j z^l
                + \sum_{i+j+l= k} \p{x\\y\\z} a_{4,ijl} x^i y^j z^l, \
    i_0=1,...n.  }
$\bv_h|_{T_{i_0}}$ is determined by
  \a{ \frac{n(k+1)(k+2)(k+3)}2 + \frac{ n(k+1)(k+2)}2=\frac{n(k+1)(k+2)(k+4)}2 }
  coefficients.
For any $\bv\in H(\div;T)$, $\Pi_h \b v\in \Lambda_k(T)$ is defined by
\begin{subequations} \label{p-j}
\an{ \label{p-j1}   \int_{e_{ij}\subset \partial T} (\Pi_h \b v-\bv) \cdot \b n_{ij}
         p_k dS & = 0 \quad \forall p_k \in P_k(e_{ij}), e_{ij}\ne e_{1,\ell},\ell\ge 2, \\
   \label{p-j2}  \int_T (\Pi_h \b v -\bv )\cdot\bn_1 p_{k-1} d \b x &=0 \quad
                \forall p_{k-1} \in P_{k-1}(T), \\
    \label{p-j3}  \int_{T_i}  (\Pi_h \b v -\bv)\cdot\bn_2 p_{k-1} d \b x &=0 \quad
                \forall p_{k-1} \in P_{k-1}(T_i), \ i=1,...n, \\
    \label{p-j4}  \int_{T_i} ( \Pi_h \b v -\b v )\cdot\bn_3 p_{k-1} d \b x &=0 \quad
                \forall p_{k-1} \in P_{k-1}(T_i), \ i=1,...n, \\
    \label{p-j5}   \int_{e_{ij}\subset T^0} [\Pi_h \b v]\cdot \b n_{ij}p_k  dS &=0
               \quad \forall p_k \in P_k(e_{ij}), \\
    \label{p-j6}  \int_{T_1} \nabla \cdot ( \Pi_h \b v|_{T_i} -\Pi_h \b v|_{T_1} ) p_{k} d \b x &=0 \quad
                \forall p_{k} \in P_k(T_1), \ i=2,...,n, \\
    \label{p-j7}  \int_{e_{1,1}}  ( \Pi_h \b v|_{e_{1,i}} -\Pi_h \b v|_{e_{1,1}} )
         \cdot \b n p_{k} d S &=0 \quad
                \forall p_{k} \in P_k(e_{1,1}), \ i=2,...,m,
 } \end{subequations}
where $e_{ij}$ is the $j$-th face triangle of $T_i$ with a fixed normal vector $\bn_{ij}$,
     $\bn_1$ is a unit vector not parallel to any internal face normal $\bn_{ij}$,
    $(\bn_1,\bn_2,\bn_3)$ forms a right-hand orthonormal system,
      $[\cdot]$ denotes the jump on a face triangle,
    $\Pi_h\b v|_{T_i}$ is understood as a polynomial vector which can be used on
    another tetrahedron $T_1$,
     $e_{1,i}\subset e_1\subset \partial T$ is a face triangle of
      $T_{i_i}$, $\b n$ is a normal vector on $e_1$,
    and $\Pi_h \b v|_{e_{1,i}}$ is extended to the whole $e_1$ as
      one polynomial.
The linear system \eqref{p-j} of equations has the following number of equations,
\a{  &\quad \ ( n+3-m)\frac{(k+1)(k+2)}2 + (2n+1)\frac{k(k+1)(k+2)}6 \\
     &\quad \ +(n-1) \frac{(k+1)(k+2)}2 + (n-1) \frac{ (k+1)(k+2)(k+3)}6 \\
     &\quad \ +(m-1)  \frac{(k+1)(k+2)}2 \\
     &= \frac{n(k+1)(k+2)(k+4)}2,
  } which is exactly the number of coefficients for a $\bv_h$ function in \eqref{v-e}.
Thus we have a square linear system.
The system has a unique solution if and only if the kernel is $\{0\}$.

Let $\bv=0$ in \eqref{p-j}.
Though $\Pi_h\bv$ is a $P_{k+1}$ polynomial,  $\Pi_h\bv\cdot \b n_{ij}$ is a $P_k$ polynomial
   when restricted on $e_{ij}$.  This can be seen by the normal format of plane equation for
   triangle $e_{ij}$.
By the first equation \eqref{p-j1}, $\Pi_h\bv\cdot \b n_{ij}= 0$ on ${e_{ij}}$, $e_{ij}\ne e_{1,\ell}$, $\ell\ge 2$.  By the seventh equation \eqref{p-j7},
   $\Pi_h\bv|_{e_{1,\ell}}\cdot \b n =\Pi_h\bv|_{e_{1,1}}\cdot \b n = 0$, $\ell\ge 2$.
In other words, \eqref{p-j7} ensures $\Pi_h\bv\cdot \b n$ is a one-piece polynomial on $e_1$
   and \eqref{p-j1} enforces it to zero.
By the sixth equation \eqref{p-j6}, $\nabla \cdot \Pi_h\bv$ is a one-piece polynomial on the whole
   $T$.
Because $\nabla \cdot \Pi_h\bv$ is continuous on inner interface triangles and is a $P_{k}(e_{ij})$
   polynomial on all the outer face triangles, by the first five equations in \eqref{p-j},
    we have
\a{ \int_T (\nabla \cdot \Pi_h\bv)^2 d \b x &=\sum_{i=1}^n
      \Big(\int_{T_i} -\Pi_h\bv\cdot \nabla(\nabla \cdot  \Pi_h\bv) d\b x
          +\int_{\partial T_i} \Pi_h\bv\cdot \b n (\nabla \cdot \Pi_h\bv) dS \Big) \\
    &=\sum_{i=1}^n
      \sum_{j=1}^3 \int_{T_i} -(\Pi_h\bv\cdot \bn_j)(\bn_j\cdot \nabla(\nabla \cdot \Pi_h\bv)) d\b x  \\
   &=0.  }
That is,
\an{\label{div0} \nabla \cdot  \Pi_h\bv=0 \quad\text{ on } \ T.  }

Starting from a corner tetrahedron $T_1$, we have its three face triangles, $e_{11}$,
   $e_{12}$ and $e_{13}$, on the boundary of $T$.
The forth face triangle $e_{14}$ of $T_1$ is shared by $T_2$.
By the selection of $\bn_1$,  the normal vector $\bn_{14}=c_1\bn_1+c_2\bn_2+c_3\bn_3$ of
  $e_{14}$ has a non zero $c_1\ne 0$.
   a 2D polynomial $p_k\in P_k(e_{14})$ can be expressed as $p_k(x_2,x_3)$, where we use
   $(x_1,x_2,x_3)$ as the coordinate variables under the system $(\bn_1, \bn_2, \bn_3)$.
Viewing this polynomial as a 3D polynomial, i.e. extending it constantly in $x_1$-direction,
   we have \a{ p_k(x_1,x_2,x_3) = p_k(x_2,x_3), \quad (x_1,x_2,x_3)\in T_1.  }
By \eqref{div0} and the third and fourth equations of \eqref{p-j},  it follows that
  \an{\nonumber 0 &=\int_{T_1} ( \nabla \cdot \Pi_h\bv )  p_k d\b x \\
      \nonumber  &= -\int_{T_1} [(\Pi_h\bv\cdot \bn_1)\partial_{x_1} p_k +
                     (\Pi_h\bv\cdot \bn_2)\partial_{x_2} p_k
            + (\Pi_h\bv\cdot \bn_3)\partial_{x_3} p_k ] d\b x \\
      \nonumber  &\quad \    +\int_{e_{14}} (\Pi_h\bv)\cdot \bn_{14} p_k dS\\
      \nonumber  &= -\int_{T_1}  (\Pi_h\bv\cdot \bn_1)\cdot 0 d\b x  + 0+0
              +\int_{e_{14}} (\Pi_h\bv)\cdot \bn_{14} p_k dS\\
     \label{F14}
    &=\int_{e_{14}} (\Pi_h\bv)\cdot \bn_{14} p_k dS \quad\forall p_k\in P_{k}(e_{14}).
   } Next, for any $p_{k-1}\in P_{k-1}(T_1)$, we let $p_k\in P_k(T_1)$ be one of its
   anti-$x_1$-derivative, i.e.,  $\partial_{x_1} p_{k}=p_{k-1}$.
  Thus, by \eqref{div0}, \eqref{p-j3},  \eqref{p-j4} and \eqref{F14}, we get
\an{ \nonumber 0 &=\int_{T_1} \nabla \cdot \Pi_h\bv p_k d\b x \\
    \nonumber    &= -\int_{T_1} [(\Pi_h\bv \cdot\bn_1 )  \partial_{x_1} p_{k}  + 0
            + 0 ] d\b x
          +\int_{e_{14}} (\Pi_h\bv)\cdot \bn_{14} p_k dS\\
    \label{T1}  &=-\int_{T_1}  (\Pi_h\bv\cdot \bn_1)   p_{k-1}   d\b x  \quad\forall p_{k-1}\in P_{k-1}(T_1).
   }

Continuing work on $T_1$, by $\nabla \cdot \Pi_h\bv=0$,  all $a_{4,ijl}=0$ in \eqref{v-e}, since
   the divergence of each such term is non-zero and independent of the divergence of
   other terms.
Thus $\Pi_h\bv |_{T_i}$ is in $[P_k(T_i)]^d$, instead of $RT_k(T_i)$.
It can be linearly expanded by the three projections on three linearly independent
   directions.
In particular, on a corner tetrahedron $T_1$ we have three outer triangles $e_{1j}$ on
   $\partial T$.
On $T_1$,
\a{ \Pi_h\bv = A \p{ \Pi_h\bv \cdot \bn_{11} \\ \Pi_h\bv \cdot \bn_{12}\\
     \Pi_h\bv \cdot \bn_{13}}= A\p{ p_1 \\ p_2\\p_3 }, }
where $p_1, p_2$ and $ p_3$ are scalar $P_k$ polynomials, and $A$ is a $3\times 3$
   scalar matrix.

By the first equation in \eqref{p-j},  $p_1$ vanishes on $e_{11}$ and
\a{ p_1 = \lambda_1 q_{k-1} \quad \text{ on } \ T_1,
} where $\lambda_1$ is a barycentric coordinate of $T_1$ (which is a linear function assuming   $0$ on $e_{11}$),  and $q_{k-1}$ is a $P_{k-1}(T)$ polynomial.
Let $p_k\in P_k(T)$ be an anti-$x$-derivative of $(\bn_{11})_1 q_{k-1}$, i.e.,
   $(\nabla p_k)_1 = (\bn_{11})_1 q_{k-1}$.
Note that $(\nabla p_k)_2$ and $(\nabla p_k)_3$ can be anything (of $y$ and $z$ functions)
  which result in zero integrals below.
By \eqref{T1}, \eqref{p-j3} and \eqref{p-j4},
   since $\nabla \cdot \Pi_h \bv=0$, we get
\a{ \int_{T_1} \lambda_1 q_{k-1}^2 d\b x &
     = \int_{T_1} \Pi_h \bv \cdot( \b n_{11} q_{k-1})  d\b x
     =  0. }
Since $\lambda_1>0$ in $T_1$,  we conclude with $q_{k-1}=0$ and $p_1=0$.
Repeating the analysis, as $p_2=0$ on $e_{12}$ and $p_3=0$ on $e_{13}$,
    we get $p_2=p_3=0$ and $\Pi_h\bv =0$ on $T_1$.

Adding the equations \eqref{F14} and \eqref{T1} to \eqref{p-j},  $T_2$ would be a new
  corner tetrahedron with three no-flux boundary triangles.
  Repeating the estimates on $T_1$, it would
   lead $\Pi_h\bv=0$ on $T_2$. Sequentially,  we obtain $\Pi_h\bv=0$ on all $T_i$, i.e.,
   on the whole $T$.

 \eqref{P1} follows \eqref{p-j2},  \eqref{p-j3} and
     \eqref{p-j4}.  \eqref{p-j1} and
     \eqref{p-j7} imply \eqref{P2}.
For a $\bv  \in H(\div; \Omega)$ and a $v\in P_k(T)$,  we have, by \eqref{p-j}, \eqref{F14}
   and \eqref{T1},
\a{ (\nabla \cdot (\bv  -\Pi_h \bv ), v)_T &=
     \sum_{i=1}^n \Big(\int_{T_i} (\bv  -\Pi_h \bv )\cdot\nabla v d\b x
            + \int_{\partial T_i} (\bv  -\Pi_h \bv )\cdot\b n v d S\Big) \\
       &= \sum_{i=1}^n  0 + \int_{\partial T} (\bv  -\Pi_h \bv )\cdot\b n v d S\Big) =0.
  }That is, \eqref{key2} holds.

It follows from (\ref{key2}) and (\ref{d-d}) that for $v=\{v_0,v_b\}\in V_h^0$
\begin{eqnarray*}
-(\nabla\cdot\bv , \;v_0)_{\T_h}&=&-(\nabla\cdot\Pi_h\bv , \;v_0)_{\T_h}\\
&=&-(\nabla\cdot\Pi_h\bv , \;v_0)_{\T_h}+\l v_b, \Pi_h\bv \cdot\bn\r_{\partial \T_h}\\
&=&(\Pi_h\bv , \;\nabla_w v)_{\T_h},
\end{eqnarray*} which proves \eqref{key1}.

Since $[P_k(T)]^3\subset \Lambda_k$ and $\Pi_h$ is uni-solvent, $\Pi_h \b v = \b v$ for all
  $\bv \in [P_k(T)]^3$.
On one size $1$ $T$, by the finite dimensional norm-equivalence and the
   shape-regularity assumption on sub-triangles,  the interpolation is stable in $L^2(T)$,
  i.e.,
\an{  \|\Pi_h \bv  \|_T \le C \|\bv \|_T.   \label{T-stable} }
After a scaling, the constant $C$ in \eqref{T-stable} remains same for all $h>0$.
It follows that
\a{ \| \Pi_h\bv  -\bv  \|^2 &\le C
    \sum_{T\in\mathcal T_h} (\| \Pi_h(\bv  -p_{k,T}) \|_T^2 + \| p_{k,T} -\bv  \|_T^2 ) \\
      &\le C
    \sum_{T\in\mathcal T_h} (C \| \bv  -p_{k,T} \|_T^2 + \| p_{k,T} -\bv  \|_T^2 ) \\
      &\le C
    \sum_{T\in\mathcal T_h} h^{2k+2}  | \bv   |_{k+1,T}^2  \\
      &=C h^{2k+2}  | \bv   |_{k+1}^2,
} where $p_{k,T}$ is the $k$-th Taylor polynomial of $\bv $ on $T$.
\end{proof}

Let $Q_0$ and $Q_b$ be the two element-wise defined $L^2$ projections onto $P_k(T)$ and $P_k(e)$ with $e\subset\partial T$ on $T$ respectively. Define $Q_hu=\{Q_0u,Q_bu\}\in V_h$. Let $\Q_h$ be the element-wise defined $L^2$ projection onto $\Lambda_k(T)$ on each element $T$.

\begin{lemma}
Let $\phi\in H^1(\Omega)$, then on any $T\in\T_h$,
\begin{eqnarray}
\nabla_w Q_h\phi =\Q_h\nabla\phi.\label{key0}
\end{eqnarray}
\end{lemma}
\begin{proof}
Using (\ref{d-d}), the definition of $\Lambda_k(T)$ and  integration by parts, we have that for
any $\bq\in \Lambda_k(T)$
\begin{eqnarray*}
(\nabla_w Q_h\phi,\bq)_T &=& -(Q_0\phi,\nabla\cdot\bq)_T
+\langle Q_b\phi,\bq\cdot\bn\rangle_{\pT}\\
&=& -(\phi,\nabla\cdot\bq)_T
+\langle \phi,\bq\cdot\bn\rangle_{\pT}\\
&=&(\nabla \phi,\bq)_T=(\Q_h\nabla\phi,\bq)_T,
\end{eqnarray*}
which implies the desired identity (\ref{key0}). We have proved the lemma.
\end{proof}

For any $v\in V_h$, let
\begin{equation}\label{3barnorm}
\3bar v\3bar^2=(\nabla_wv,\nabla_wv)_{\T_h}.
\end{equation}
We introduce a discrete $H^1$ semi-norm as follows:
\begin{equation}\label{norm}
\|v\|_{1,h} = \left( \sum_{T\in\T_h}\left(\|\nabla
v_0\|_T^2+h_T^{-1} \|  v_0-v_b\|^2_\pT\right) \right)^{\frac12}.
\end{equation}
It is easy to see that $\|v\|_{1,h}$ define a norm in $V_h^0$. The following lemma indicates that $\|\cdot\|_{1,h}$ is equivalent
to the $\3bar\cdot\3bar$ in (\ref{3barnorm}).

\begin{lemma} There exist two positive constants $C_1$ and $C_2$ such
that for any $v=\{v_0,v_b\}\in V_h$, we have
\begin{equation}\label{happy}
C_1 \|v\|_{1,h}\le \3bar v\3bar \leq C_2 \|v\|_{1,h}.
\end{equation}
\end{lemma}

\medskip

\begin{proof}
For any $v=\{v_0,v_b\}\in V_h$, it follows from the definition of
weak gradient (\ref{d-d}) and integration by parts that
\begin{eqnarray}\label{n-1}
(\nabla_wv,\bq)_T=(\nabla v_0,\bq)_T+\l v_b-v_0,
\bq\cdot\bn\r_\pT,\quad \forall \bq\in \Lambda_k(T).
\end{eqnarray}
By letting $\bq=\nabla_w v$ in (\ref{n-1}) we arrive at
\begin{eqnarray*}
(\nabla_wv,\nabla_w v)_T=(\nabla v_0,\nabla_w v)_T+\l v_b-v_0,
\nabla_w v\cdot\bn\r_\pT.
\end{eqnarray*}
From the trace inequality (\ref{trace}) and the inverse inequality
we have
\begin{eqnarray*}
\|\nabla_wv\|^2_T &\le& \|\nabla v_0\|_T \|\nabla_w v\|_T+ \|
v_0-v_b\|_\pT \|\nabla_w v\|_\pT\\
&\le& \|\nabla v_0\|_T \|\nabla_w v\|_T+ Ch_T^{-1/2}\|
v_0-v_b\|_\pT \|\nabla_w v\|_T,
\end{eqnarray*}
which implies
$$
\|\nabla_w v\|_T \le C \left(\|\nabla v_0\|_T +h_T^{-1/2}\|v_0-v_b\|_\pT\right),
$$
and consequently
$$\3bar v\3bar \leq C_2 \|v\|_{1,h}.$$

Next we will prove $C_1 \|v\|_{1,h}\le \3bar v\3bar $.
The construction of $\Lambda_k(T)$ implies there exists $\bq_0\in \Lambda_h(T)$ such that
\begin{equation}\label{2e}
(\nabla v_0,\bq_0)_T=0, \ \
    \ \ \l v_b-v_0, \bq_0\cdot\bn\r_\pT=\|v_0-v_b\|_\pT^2,
\end{equation}
and
\begin{equation}
\|\bq_0\|_T \le C h_T^{1/2} \| v_b-v_0 \|_e.\label{22e}
\end{equation}
Letting $\bq=\bq_0$ in (\ref{n-1}), we get
\begin{equation}\label{n3}
(\nabla_wv,\bq_0)_T=\|v_b-v_0\|^2_e.
\end{equation}
It follows from Cauchy-Schwarz inequality and (\ref{22e}) that
\[
\|v_b-v_0\|^2_e\le C\|\nabla_w v\|_T\|\bq_0\|_T
 \le Ch_T^{1/2}\|\nabla_w v\|_T\|v_0-v_b\|_e,
\]
which implies
\begin{equation}\label{n4}
h_T^{-1/2}\|v_0-v_b\|_\pT\le C\|\nabla_w v\|_T.
\end{equation}
It follows from the trace inequality, the inverse inequality and (\ref{n4}),
$$
\|\nabla v_0\|_T^2 \leq \|\nabla_w v\|_T \|\nabla v_0\|_T
+Ch_T^{-1/2}\| v_0-v_b\|_\pT \|\nabla v_0\|_T\le C\|\nabla_w v\|_T \|\nabla v_0\|_T.
$$
Combining the above estimate and (\ref{n4}),
   by the definition \eqref{norm},
 we prove  the lower bound of (\ref{happy}) and complete the proof of the lemma.
\end{proof}

\medskip

\begin{lemma}
The weak Galerkin finite element scheme (\ref{wg}) has a unique
solution.
\end{lemma}

\smallskip

\begin{proof}
If $u_h^{(1)}$ and $u_h^{(2)}$ are two solutions of (\ref{wg}), then
$\varepsilon_h=u_h^{(1)}-u_h^{(2)}\in V_h^0$ would satisfy the following equation
$$
(\nabla_w \varepsilon_h,\nabla_w v)=0,\qquad\forall v\in V_h^0.
$$
 Then by letting $v=\varepsilon_h$ in the above
equation we arrive at
$$
\3bar \varepsilon_h\3bar^2 = (\nabla_w \varepsilon_h,\nabla_w \varepsilon_h)=0.
$$
It follows from (\ref{happy}) that $\|\varepsilon_h\|_{1,h}=0$. Since $\|\cdot\|_{1,h}$ is a norm in $V_h^0$, one has $\varepsilon_h=0$.
 This completes the proof of the lemma.
\end{proof}

\section{Error Equations}
Let $\epsilon_h=Q_hu-u_h$. Next we derive two error equations that $\epsilon_h$ satisfies. One will be used in energy norm error analysis and another one for $L^2$ error estimate.

\begin{lemma}
For any $v\in V_h^0$, the following error equation holds true
\begin{eqnarray}
(\nabla_w \epsilon_h,\nabla_wv)_{\T_h}=\ell(u,v),\label{ee}
\end{eqnarray}
where
\begin{eqnarray*}
\ell(u,v)&=&(\bbQ_h\nabla u-\Pi_h\nabla u, \nabla_w v)_{\T_h}
\end{eqnarray*}
\end{lemma}

\begin{proof}
For $v=\{v_0,v_b\}\in V_h^0$, testing (\ref{pde}) by  $v_0$  and using (\ref{key1}),  we arrive at
\begin{equation}\label{m1}
(f, v_0)=-(\nabla\cdot\nabla u, v_0)_{\T_h}=(\Pi_h\nabla u, \nabla_w v)_{\T_h}.
\end{equation}
It follows from (\ref{key0}) and (\ref{m1})
\begin{equation}\label{j2}
(\nabla_wQ_h u, \nabla_w v)_{\T_h}=(f, v_0)+\ell(u,v).
\end{equation}
The error equation follows from subtracting (\ref{wg}) from (\ref{j2}),
\begin{eqnarray*}
(\nabla_w\epsilon_h,\nabla_wv)_{\T_h}=\ell(u,v)\quad \forall v\in V_h^0.
\end{eqnarray*}
This completes the proof of the lemma.
\end{proof}

\begin{lemma}
For any $v\in V_h^0$, the following error equation holds true
\begin{eqnarray}
(\nabla_w\epsilon_h,\nabla_wv)_{\T_h}=\ell_1(u,v),\label{ee1}
\end{eqnarray}
where
\begin{eqnarray*}
\ell_1(u,v)&=& \langle (\nabla u-\Q_h\nabla u)\cdot\bn,v_0-v_b\rangle_{\partial T_h}.
\end{eqnarray*}
\end{lemma}

\begin{proof}
For $v=\{v_0,v_b\}\in V_h^0$, testing (\ref{pde}) by  $v_0$  and using integration by parts and the fact that
$\sum_{T\in\T_h}\langle \nabla u\cdot\bn, v_b\rangle_\pT=0$,  we arrive at
\begin{equation}\label{mm1}
(\nabla u,\nabla v_0)_{\T_h}- \langle
\nabla u\cdot\bn,v_0-v_b\rangle_{\partial T_h}=(f,v_0).
\end{equation}

It follows from integration by parts, (\ref{d-d}) and (\ref{key0})  that
\begin{eqnarray}
(\nabla u,\nabla v_0)_{\T_h}&=&(\Q_h\nabla  u,\nabla v_0)_{\T_h}\nonumber\\
&=&-(v_0,\nabla\cdot (\Q_h\nabla u))_{\T_h}+\langle v_0, \Q_h\nabla u\cdot\bn\rangle_{\partial\T_h}\nonumber\\
&=&(\Q_h\nabla u, \nabla_w v)_{\T_h}+\langle v_0-v_b,\Q_h\nabla u\cdot\bn\rangle_{\partial\T_h}\nonumber\\
&=&( \nabla_w Q_hu, \nabla_w v)_{\T_h}+\langle v_0-v_b,\Q_h\nabla u\cdot\bn\rangle_{\partial\T_h}.\label{j1}
\end{eqnarray}
Combining (\ref{mm1}) and (\ref{j1}) gives
\begin{eqnarray}
(\nabla_w Q_hu,\nabla_w v)_{\T_h}&=&(f,v_0)+\ell_1(u,v).\label{jj2}
\end{eqnarray}
The error equation follows from subtracting (\ref{wg}) from (\ref{jj2}),
\begin{eqnarray*}
(\nabla_w\epsilon_h,\nabla_wv)_{\T_h}=\ell_1(u,v)\quad \forall v\in V_h^0.
\end{eqnarray*}
This completes the proof of the lemma.
\end{proof}

\section{Error Estimates}

For any function $\varphi\in H^1(T)$, the following trace
inequality holds true (see \cite{wymix} for details):
\begin{equation}\label{trace}
\|\varphi\|_{e}^2 \leq C \left( h_T^{-1} \|\varphi\|_T^2 + h_T
\|\nabla \varphi\|_{T}^2\right).
\end{equation}

\begin{theorem} Let $u_h\in V_h$ be the weak Galerkin finite element solution of (\ref{wg}). Assume the exact solution $u\in H^{k+2}(\Omega)$. Then,
there exists a constant $C$ such that
\begin{equation}\label{err1}
\3bar Q_hu-u_h\3bar \le Ch^{k+1}|u|_{k+2}.
\end{equation}
\end{theorem}
\begin{proof}
Letting $v=\epsilon_h$ in (\ref{ee}) gives
\begin{eqnarray}
\3bar \epsilon_h\3bar^2&=&\ell(u,\epsilon_h).\label{eee1}
\end{eqnarray}
The definitions of $\bbQ_h$ and $\Pi_h$ imply
\begin{eqnarray}
|\ell(u,\epsilon_h)|&=&|(\bbQ_h\nabla u-\Pi_h\nabla u, \nabla_w \epsilon_h)_{\T_h}|\nonumber\\
&\le&(\sum_T\|\bbQ_h\nabla u-\Pi_h\nabla u\|_T)^{1/2}\3bar\epsilon_h\3bar\nonumber\\
&\le&(\sum_T\|\bbQ_h\nabla u-\nabla u+\nabla u-\Pi_h\nabla u\|_T)^{1/2}\3bar\epsilon_h\3bar\nonumber\\
&\le& Ch^{k+1}|u|_{k+2}\3bar \epsilon_h\3bar.\label{eee3}
\end{eqnarray}
Combining (\ref{eee1}) and (\ref{eee3}), we arrive
\[
\3bar \epsilon_h\3bar \le Ch^{k+1}|u|_{k+2},
\]
which completes the proof of the theorem.
\end{proof}


The standard duality argument is used to obtain $L^2$ error estimate.
Recall $\epsilon_h=\{\epsilon_0,\epsilon_b\}=Q_hu-u_h$.
The considered dual problem seeks $\Phi\in H_0^1(\Omega)$ satisfying
\begin{eqnarray}
-\Delta\Phi&=& \epsilon_0\quad
\mbox{in}\;\Omega.\label{dual}
\end{eqnarray}
Assume that the following $H^{2}$ regularity holds
\begin{equation}\label{reg}
\|\Phi\|_2\le C\|\epsilon_0\|.
\end{equation}

\begin{theorem} Let $u_h\in V_h$ be the weak Galerkin finite element solution of (\ref{wg}). Assume that the
exact solution $u\in H^{k+2}(\Omega)$ and (\ref{reg}) holds true.
 Then, there exists a constant $C$ such that for $k\ge 1$
\begin{equation}\label{err2}
\|Q_0u-u_0\| \le Ch^{k+2}|u|_{k+2}.
\end{equation}
\end{theorem}

\begin{proof}
Testing (\ref{dual}) by $\epsilon_0$ and using the fact that $\sum_{T\in\T_h}\langle \nabla
\Phi\cdot\bn, \epsilon_b\rangle_\pT=0$ give
\begin{eqnarray}\nonumber
\|\epsilon_0\|^2&=&-(\Delta\Phi,\epsilon_0)\\
&=&(\nabla \Phi,\ \nabla \epsilon_0)_{\T_h}-\l
\nabla\Phi\cdot\bn,\ \epsilon_0- \epsilon_b\r_{\pT_h}.\label{jw.08}
\end{eqnarray}
Setting $u=\Phi$ and $v=\epsilon_h$ in (\ref{j1})  yields
\begin{eqnarray}
(\nabla\Phi,\;\nabla \epsilon_0)_{\T_h}=(\nabla_w Q_h\Phi,\;\nabla_w \epsilon_h)_{\T_h}+\l
\Q_h\nabla\Phi\cdot\bn,\ \epsilon_0-\epsilon_b\r_{\pT_h}.\label{j1-new}
\end{eqnarray}
Substituting (\ref{j1-new}) into (\ref{jw.08}) and using (\ref{ee1}) yield
\begin{eqnarray}
\|\epsilon_0\|^2&=&(\nabla_w \epsilon_h,\ \nabla_w Q_h\Phi)_{\T_h}-\l
(\nabla\Phi-\Q_h\nabla\Phi)\cdot\bn,\ \epsilon_0-\epsilon_b\r_{\pT_h}\nonumber\\
&=&\ell_1(u,Q_h\Phi)-\ell_1(\Phi, \epsilon_h).\label{m2}
\end{eqnarray}
Next we estimate the two terms on the right hand side of (\ref{m2}). Using the Cauchy-Schwarz inequality, the trace inequality (\ref{trace}) and the definitions of $Q_h$ and $\Pi_h$
we obtain
\begin{eqnarray*}
|\ell_1(u,Q_h\Phi)|&\le&\left| \langle (\nabla u-\Q_h\nabla
u)\cdot\bn,\;
Q_0\Phi-Q_b\Phi\rangle_{\pT_h} \right|\\
&\le& \left(\sum_{T\in\T_h}\|\nabla u-\Q_h\nabla
u\|^2_\pT\right)^{1/2}
\left(\sum_{T\in\T_h}\|Q_0\Phi-Q_b\Phi\|^2_\pT\right)^{1/2}\nonumber \\
&\le& C\left(\sum_{T\in\T_h}h\|\nabla u-\Q_h\nabla
u\|^2_\pT\right)^{1/2}
\left(\sum_{T\in\T_h}h^{-1}\|Q_0\Phi-\Phi\|^2_\pT\right)^{1/2} \nonumber\\
&\le&  Ch^{k+2}|u|_{k+2}|\Phi|_2.\nonumber
\end{eqnarray*}

Using the Cauchy-Schwarz inequality, the trace inequality (\ref{trace}), (\ref{happy}) and (\ref{err1}), we have
\begin{eqnarray*}
|\ell_1(\Phi,\epsilon_h)|&=&\left|\sum_{T\in\T_h}\langle (\nabla \Phi-\Q_h\nabla
\Phi)\cdot\bn, \epsilon_0-\epsilon_b\rangle_\pT\right|\\
&\le & C \sum_{T\in\T_h}\|\nabla \Phi-\Q_h\nabla \Phi\|_{\pT}
\|\epsilon_0-\epsilon_b\|_\pT\nonumber\\
&\le & C \left(\sum_{T\in\T_h}h_T\|\nabla \Phi-\Q_h\nabla \Phi\|_{\pT}^2\right)^{\frac12}
\left(\sum_{T\in\T_h}h_T^{-1}\|\epsilon_0-\epsilon_b\|_\pT^2\right)^{\frac12}\\
&\le & Ch^{k+2}|u|_{k+2}|\Phi|_{2}.
\end{eqnarray*}
Combining the two estimates above with (\ref{m2}) yields
$$
\|\epsilon_0\|^2 \leq C h^{k+2}|u|_{k+2} \|\Phi\|_2.
$$
It follows from the above inequality and
the regularity assumption (\ref{reg}),
 $$
\|\epsilon_0\|\leq C h^{k+2}|u|_{k+2}.
$$
We have completed the proof.
\end{proof}

\section{Numerical Experiments}\label{Section:numerical-experiments}

We solve the Poisson problem \eqref{pde}-\eqref{bc} on the unit square domain with the exact
  solution
\an{ \label{s-1} u=\sin(\pi x)\sin(\pi y).
  } We first use the uniform square grids shown in Figure \ref{g-square}.
We then compute the problem with several
We list the computational results in Table \ref{t0}.
As proved,   we have one order of super-convergence for both
   $L^2$  errors and $H^1$-like errors.

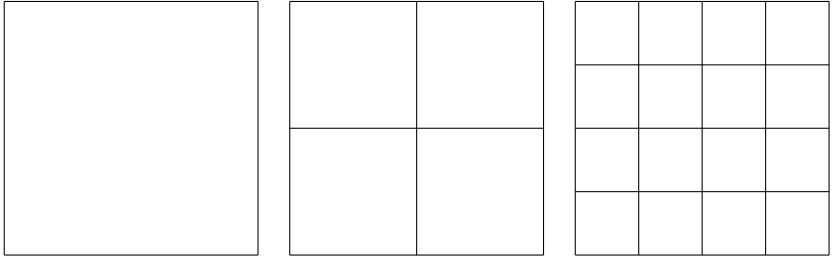
\begin{figure}[h!]
\begin{center}
 \setlength\unitlength{1.2pt}
    \begin{picture}(260,84)(0,2)
    \put(0,0){\begin{picture}(110,110)(0,0)
       \multiput(0,0)(80,0){2}{\line(0,1){80}}  \multiput(0,0)(0,80){2}{\line(1,0){80}}
      \end{picture}}
    \put(90,0){\begin{picture}(110,110)(0,0)
       \multiput(0,0)(40,0){3}{\line(0,1){80}}  \multiput(0,0)(0,40){3}{\line(1,0){80}}
      \end{picture}}
    \put(180,0){\begin{picture}(110,110)(0,0)
       \multiput(0,0)(20,0){5}{\line(0,1){80}}  \multiput(0,0)(0,20){5}{\line(1,0){80}}
      \end{picture}}
    \end{picture}
    \end{center}
\caption{  The first three levels of wedge grids used in Table \ref{t3}. }
\label{g-square}
\end{figure}

\begin{table}[h!]
  \centering   \renewcommand{\arraystretch}{1.05}
  \caption{ Error profiles and convergence rates on square grids shown in Figure \ref{g-square} for \eqref{s-1}. }
\label{t0}
\begin{tabular}{c|cc|cc}
\hline
level & $\|Q_h u-  u_h \|_0 $  &rate &  $\3bar Q_h u- u_h \3bar $ &rate   \\
\hline
 &\multicolumn{4}{c}{by the $P_0$-$P_0$($\Lambda_0$) WG element} \\ \hline
 6&   0.1101E-02 &  1.99&   0.1988E+00 &  0.99\\
 7&   0.2756E-03 &  2.00&   0.9951E-01 &  1.00\\
 8&   0.6892E-04 &  2.00&   0.4977E-01 &  1.00\\
\hline
 &\multicolumn{4}{c}{by the $P_1$-$P_1$($\Lambda_1$) WG element} \\ \hline
 6&   0.2722E-04 &  2.99&   0.6952E-02 &  2.00\\
 7&   0.3407E-05 &  3.00&   0.1739E-02 &  2.00\\
 8&   0.4261E-06 &  3.00&   0.4347E-03 &  2.00\\
 \hline
 &\multicolumn{4}{c}{by the $P_2$-$P_2$($\Lambda_2$) WG element} \\ \hline
 6&   0.8248E-06 &  4.00&   0.3106E-03 &  3.00\\
 7&   0.5156E-07 &  4.00&   0.3884E-04 &  3.00\\
 8&   0.3313E-08 &  3.96&   0.4855E-05 &  3.00\\
\hline
 &\multicolumn{4}{c}{by the $P_3$-$P_3$($\Lambda_3$) WG element} \\ \hline
 5&   0.6585E-06 &  4.99&   0.1674E-03 &  3.99\\
 6&   0.2060E-07 &  5.00&   0.1047E-04 &  4.00\\
 7&   0.6700E-09 &  4.94&   0.6548E-06 &  4.00\\
 \hline
 &\multicolumn{4}{c}{by the $P_4$-$P_4$($\Lambda_4$) WG element} \\ \hline
 3&   0.1110E-03 &  5.93&   0.9355E-02 &  4.93\\
 4&   0.1752E-05 &  5.99&   0.2957E-03 &  4.98\\
 5&   0.2765E-07 &  5.99&   0.9266E-05 &  5.00\\
 \hline
\end{tabular}%
\end{table}%

We compute the solution \eqref{s-1} again on a type of quadrilateral grids, shown
   in Figure \ref{g-4}.
Here to avoid convergence to parallelograms under the nest refinement of quadrilaterals,
  we fix the shape of quadrilaterals on all levels of grids.
We list the computation in Table \ref{t1}.
Again, the data confirm the theoretic convergence rates.

\begin{figure}[htb]\begin{center}
\includegraphics[width=1.4in]{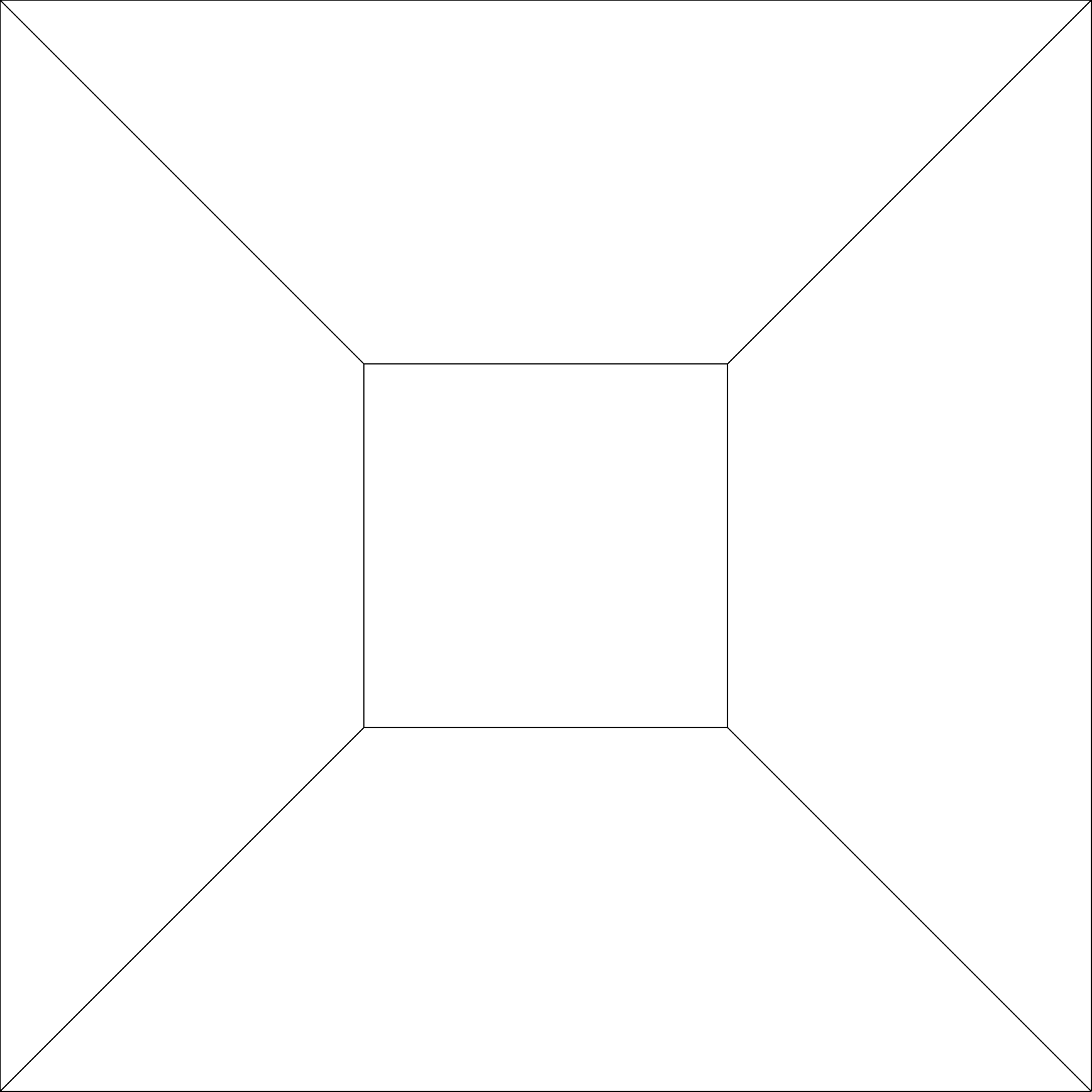} \
\includegraphics[width=1.4in]{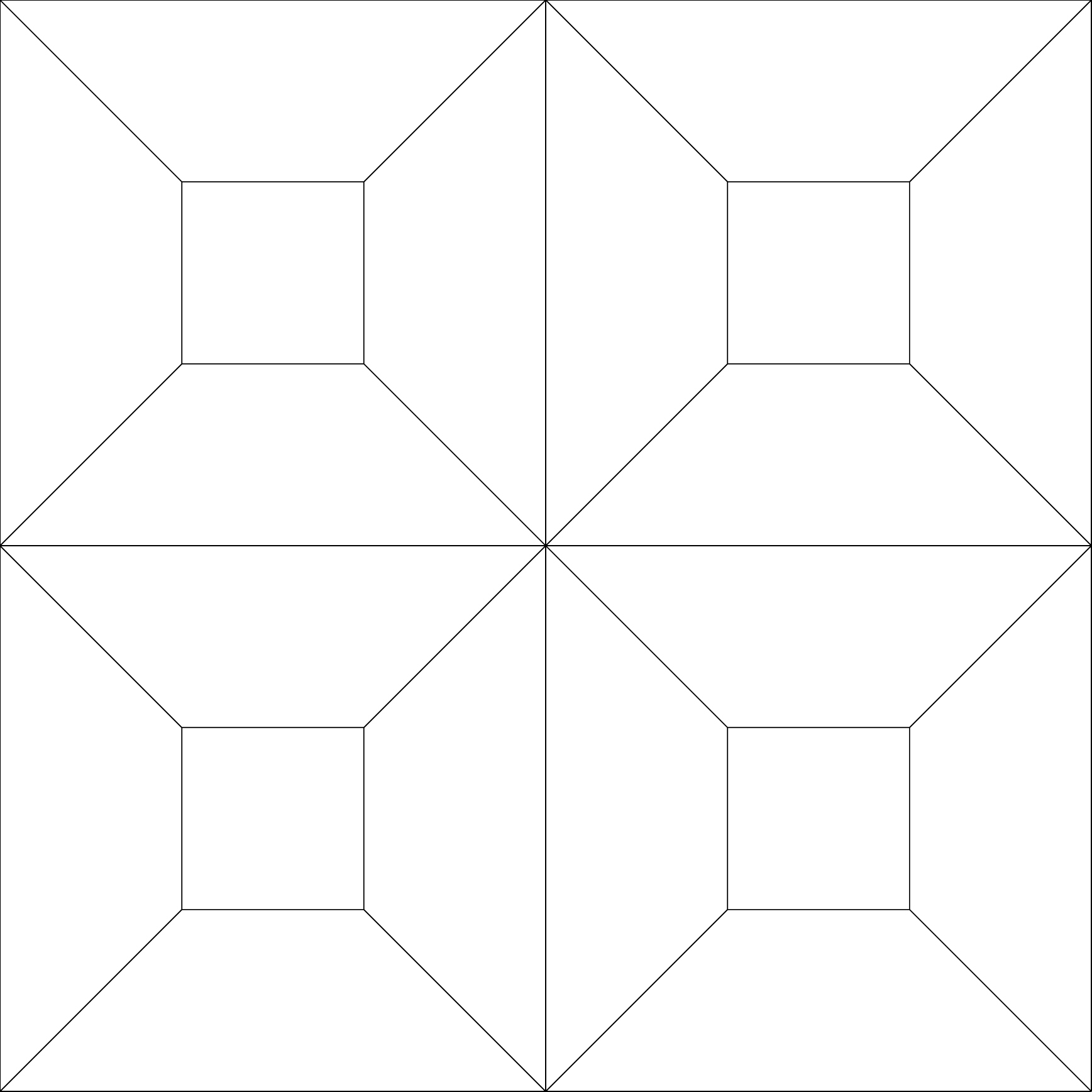} \
\includegraphics[width=1.4in]{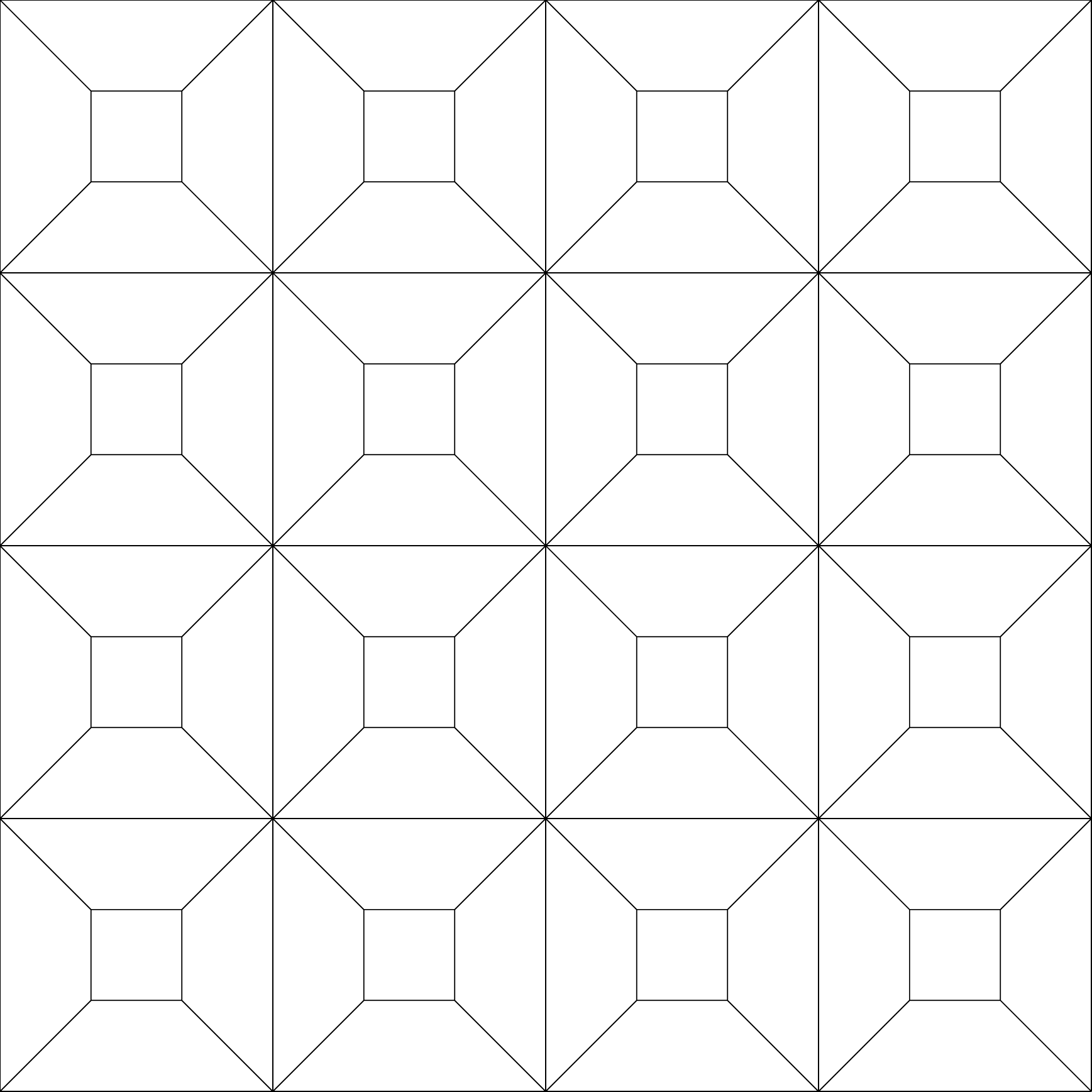}

\caption{The first three levels of grids, for Table \ref{t1}.  }
\label{g-4}
\end{center}
\end{figure}

\begin{table}[h!]
  \centering   \renewcommand{\arraystretch}{1.05}
  \caption{ Error profiles and convergence rates on quadrilateral
      grids shown in Figure \ref{g-4} for \eqref{s-1}. }
\label{t1}
\begin{tabular}{c|cc|cc}
\hline
level & $\|Q_h u-  u_h \|_0 $  &rate &  $\3bar Q_h u- u_h \3bar $ &rate   \\
\hline
 &\multicolumn{4}{c}{by the $P_0$-$P_0$($\Lambda_0$) WG element} \\ \hline
 6&   0.2491E-03 &  2.00&   0.9033E-01 &  1.00 \\
 7&   0.6231E-04 &  2.00&   0.4518E-01 &  1.00 \\
 8&   0.1558E-04 &  2.00&   0.2259E-01 &  1.00 \\
\hline
 &\multicolumn{4}{c}{by the $P_1$-$P_1$($\Lambda_1$) WG element} \\ \hline
 6&   0.2722E-04 &  2.99&   0.6952E-02 &  2.00 \\
 7&   0.3407E-05 &  3.00&   0.1739E-02 &  2.00 \\
 8&   0.4261E-06 &  3.00&   0.4347E-03 &  2.00 \\
 \hline
 &\multicolumn{4}{c}{by the $P_2$-$P_2$($\Lambda_2$) WG element} \\ \hline
 5&   0.9093E-06 &  4.00&   0.3256E-03 &  3.00 \\
 6&   0.5686E-07 &  4.00&   0.4071E-04 &  3.00 \\
 7&   0.3554E-08 &  4.00&   0.5090E-05 &  3.00 \\
\hline
 &\multicolumn{4}{c}{by the $P_3$-$P_3$($\Lambda_3$) WG element} \\ \hline
 2&   0.7967E-03 &  5.22&   0.4984E-01 &  4.34 \\
 3&   0.2629E-04 &  4.92&   0.3181E-02 &  3.97 \\
 4&   0.8342E-06 &  4.98&   0.1998E-03 &  3.99 \\
 5&   0.2618E-07 &  4.99&   0.1251E-04 &  4.00 \\
 \hline
\end{tabular}%
\end{table}%

\begin{figure}[htb]\begin{center}
\includegraphics[width=1.4in]{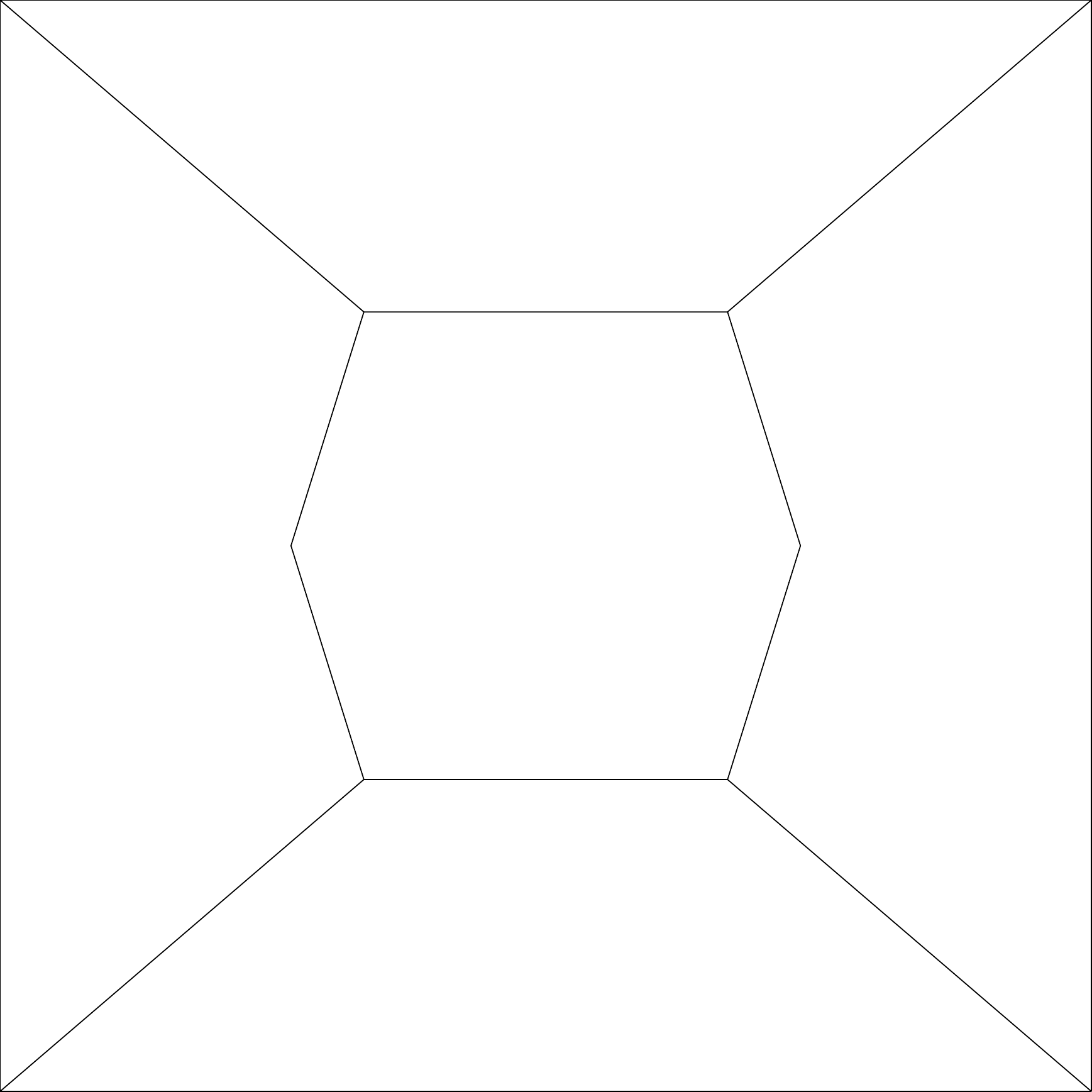} \
\includegraphics[width=1.4in]{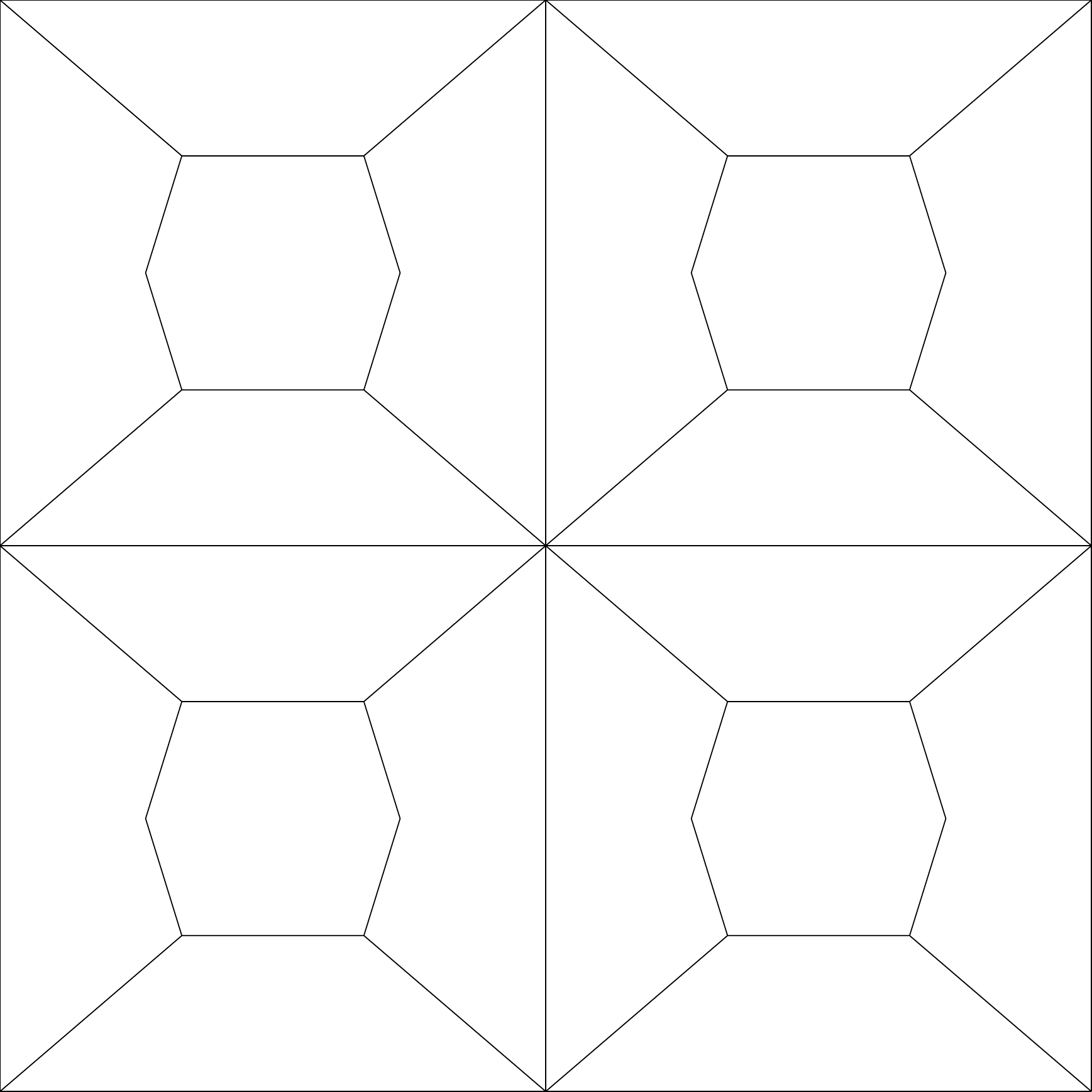} \
\includegraphics[width=1.4in]{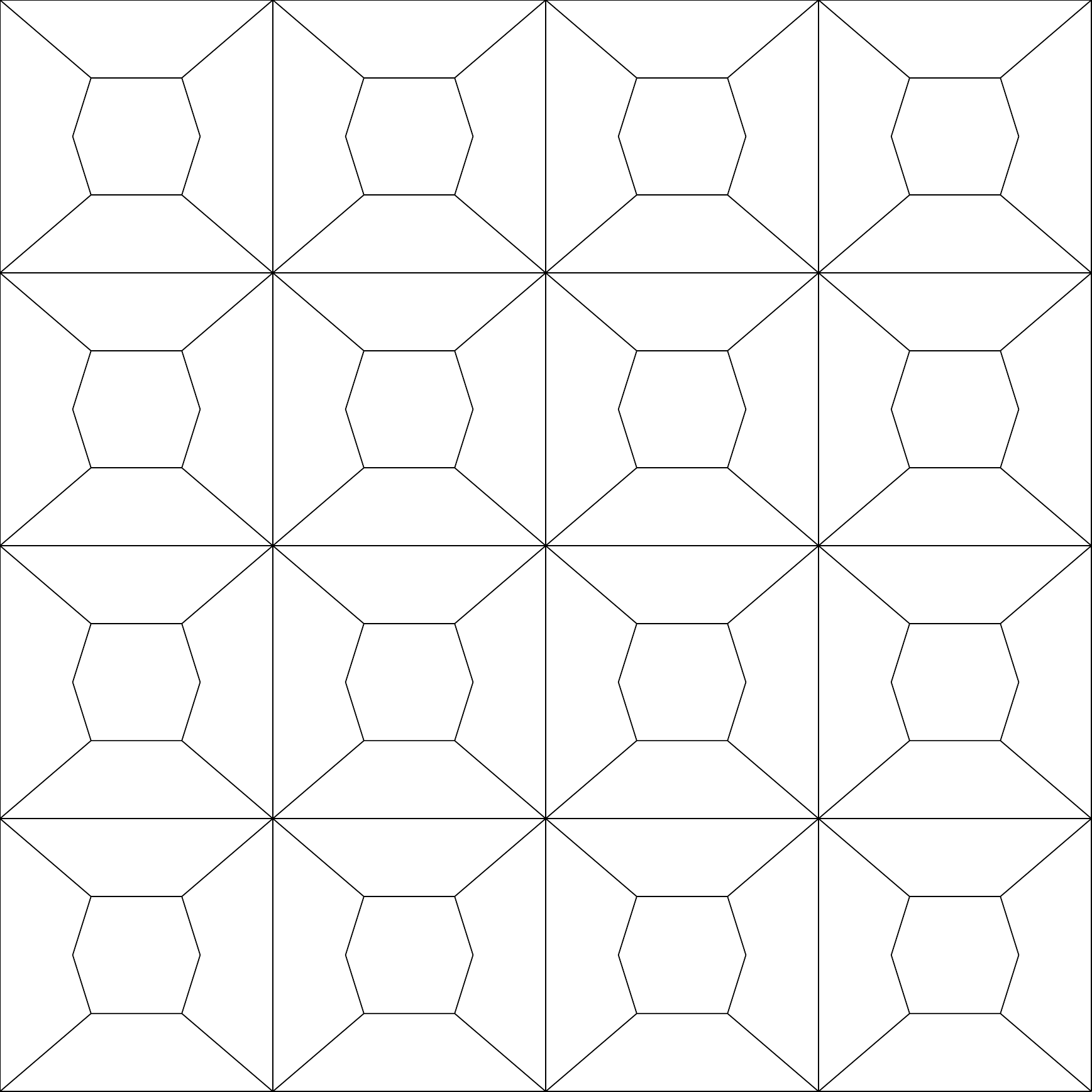}

\caption{The first three levels of quadrilateral-hexagon grids, for Table \ref{t2}.  }
\label{g-6}
\end{center}
\end{figure}

Next we solve the same problem \eqref{s-1} on a type of grids with quadrilaterals and hexagons,
 shown in Figure \ref{g-6}.
We list the result of computation in Table \ref{t2} where we obtain
   one order of superconvergence in all cases.

\begin{table}[h!]
  \centering   \renewcommand{\arraystretch}{1.05}
  \caption{ Error profiles and convergence rates on polygonal grids shown in Figure \ref{g-6} for \eqref{s-1}. }
\label{t2}
\begin{tabular}{c|cc|cc}
\hline
level & $\|Q_h u-  u_h \|_0 $  &rate &  $\3bar Q_h u- u_h \3bar $ &rate   \\
\hline
 &\multicolumn{4}{c}{by the $P_0$-$P_0$($\Lambda_0$) WG element} \\ \hline
 6&   0.1892E-03 &  2.00&   0.8731E-01 &  1.00  \\
 7&   0.4731E-04 &  2.00&   0.4367E-01 &  1.00  \\
 8&   0.1183E-04 &  2.00&   0.2184E-01 &  1.00  \\
\hline
 &\multicolumn{4}{c}{by the $P_1$-$P_1$($\Lambda_1$) WG element} \\ \hline
 6&   0.3602E-05 &  3.00&   0.1791E-02 &  2.00 \\
 7&   0.4504E-06 &  3.00&   0.4477E-03 &  2.00 \\
 8&   0.5631E-07 &  3.00&   0.1119E-03 &  2.00 \\
 \hline
 &\multicolumn{4}{c}{by the $P_2$-$P_2$($\Lambda_2$) WG element} \\ \hline
 6&   0.1850E-07 &  4.00&   0.1655E-04 &  3.00  \\
 7&   0.1156E-08 &  4.00&   0.2068E-05 &  3.00  \\
 8&   0.7299E-10 &  3.99&   0.2586E-06 &  3.00  \\
\hline
 &\multicolumn{4}{c}{by the $P_3$-$P_3$($\Lambda_3$) WG element} \\ \hline
 5&   0.7478E-08 &  5.00&   0.4103E-05 &  4.00  \\
 6&   0.2339E-09 &  5.00&   0.2565E-06 &  4.00  \\
 7&   0.7941E-11 &  4.88&   0.1603E-07 &  4.00 \\
 \hline
\end{tabular}%
\end{table}%

\begin{figure}[h!]
\begin{center}
 \setlength\unitlength{1pt}
    \begin{picture}(320,118)(0,3)
    \put(0,0){\begin{picture}(110,110)(0,0)
       \multiput(0,0)(80,0){2}{\line(0,1){80}}  \multiput(0,0)(0,80){2}{\line(1,0){80}}
       \multiput(0,80)(80,0){2}{\line(1,1){20}} \multiput(0,80)(20,20){2}{\line(1,0){80}}
       \multiput(80,0)(0,80){2}{\line(1,1){20}}  \multiput(80,0)(20,20){2}{\line(0,1){80}}
    \put(80,0){\line(-1,1){80}}
      \end{picture}}
    \put(110,0){\begin{picture}(110,110)(0,0)
       \multiput(0,0)(40,0){3}{\line(0,1){80}}  \multiput(0,0)(0,40){3}{\line(1,0){80}}
       \multiput(0,80)(40,0){3}{\line(1,1){20}} \multiput(0,80)(10,10){3}{\line(1,0){80}}
       \multiput(80,0)(0,40){3}{\line(1,1){20}}  \multiput(80,0)(10,10){3}{\line(0,1){80}}
    \put(80,0){\line(-1,1){80}}
       \multiput(40,0)(40,40){2}{\line(-1,1){40}}
      \end{picture}}
    \put(220,0){\begin{picture}(110,110)(0,0)
       \multiput(0,0)(20,0){5}{\line(0,1){80}}  \multiput(0,0)(0,20){5}{\line(1,0){80}}
       \multiput(0,80)(20,0){5}{\line(1,1){20}} \multiput(0,80)(5,5){5}{\line(1,0){80}}
       \multiput(80,0)(0,20){5}{\line(1,1){20}}  \multiput(80,0)(5,5){5}{\line(0,1){80}}
    \put(80,0){\line(-1,1){80}}
       \multiput(40,0)(40,40){2}{\line(-1,1){40}}

       \multiput(20,0)(60,60){2}{\line(-1,1){20}}   \multiput(60,0)(20,20){2}{\line(-1,1){60}}
      \end{picture}}

    \end{picture}
    \end{center}
\caption{  The first three levels of wedge grids used in Table \ref{t3}. }
\label{grid3}
\end{figure}
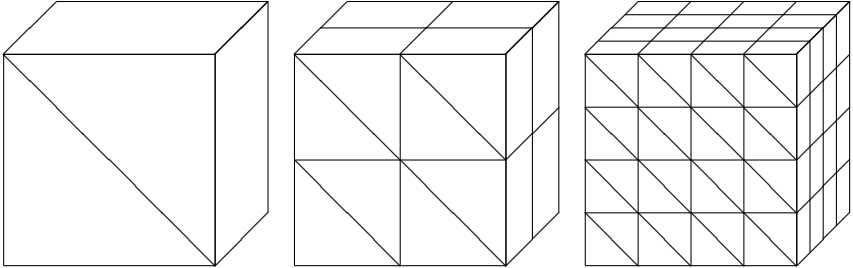

Lastly, we solve a 3D problem \eqref{pde}--\eqref{bc} on the unit cube domain
  $\Omega=(0,1)^3$ with the exact
  solution
\an{ \label{s-2}  u&=\sin(\pi x) \sin(\pi y)\sin(\pi z). }

 Here we use a uniform wedge-type (polyhedron with 2 triangle faces and 3 rectangle faces)
   grids,  shown in Figure \ref{grid3}.   Here each wedge is subdivided in to three
   tetrahedrons with three rectangular faces being cut in to two triangles,
    when defining piecewise $RT_k$
  weak gradient space $\Lambda_k$.
 The results are listed in Table \ref{t3},  confirming the one order superconvergence in
  the two norms for all polynomial-degree $k$ elements.

\begin{table}[h!]
  \centering   \renewcommand{\arraystretch}{1.05}
  \caption{ Error profiles and convergence rates on grids shown in Figure \ref{grid3} for \eqref{s-2}. }
\label{t3}
\begin{tabular}{c|cc|cc}
\hline

level & $\|Q_h u-  u_h \|_0 $  &rate &  $\3bar Q_h u- u_h \3bar $ &rate   \\
\hline
 &\multicolumn{4}{c}{by the $P_0$-$P_0$($\Lambda_0$) WG element} \\ \hline
 5&     0.0010162&2.0&     0.1241692&1.0 \\
 6&     0.0002548&2.0&     0.0621798&1.0 \\
 7&     0.0000637&2.0&     0.0311019&1.0 \\
\hline
 &\multicolumn{4}{c}{by the $P_1$-$P_1$($\Lambda_1$) WG element} \\ \hline
 4&     0.0011585&2.9&     0.1242568&2.0 \\
 5&     0.0001470&3.0&     0.0312093&2.0 \\
 6&     0.0000184&3.0&     0.0078116&2.0 \\
 \hline
 &\multicolumn{4}{c}{by the $P_2$-$P_2$($\Lambda_0$) WG element} \\ \hline
 4&     0.0001439&4.0&     0.0284085&3.0 \\
 5&     0.0000090&4.0&     0.0035641&3.0 \\
 6&     0.0000006&4.0&     0.0004459&3.0 \\
\hline
 &\multicolumn{4}{c}{by the $P_3$-$P_3$($\Lambda_0$) WG element} \\ \hline
 3&     0.0004607&4.9&     0.0770284&3.9 \\
 4&     0.0000148&5.0&     0.0048815&4.0 \\
 5&     0.0000005&5.0&     0.0003062&4.0 \\
 \hline
\end{tabular}%
\end{table}%

\end{document}